\documentclass[a4paper,english,fontsize=10pt,parskip=half,abstracton]{scrartcl}
\usepackage{babel}
\usepackage[utf8]{inputenc}
\usepackage[T1]{fontenc}
\usepackage[a4paper,left=20mm,right=20mm,top=30mm,bottom=30mm]{geometry}
\usepackage{amsmath}
\usepackage{amsthm}
\usepackage{amssymb}
\usepackage{aliascnt}
\usepackage{enumerate}
\usepackage[bookmarks=false,
            pdftitle={Cartan matrices and Brauer's k(B)-Conjecture IV},
            pdfauthor={Benjamin Sambale},
            pdfkeywords={},
            pdfstartview={FitH}]{hyperref}

\newtheorem{Theorem}{Theorem} 
\newaliascnt{Lemma}{Theorem}
\newtheorem{Lemma}[Lemma]{Lemma}
\aliascntresetthe{Lemma}
\newaliascnt{Proposition}{Theorem}
\newtheorem{Proposition}[Proposition]{Proposition}
\aliascntresetthe{Proposition}
\newaliascnt{Corollary}{Theorem}
\newtheorem{Corollary}[Corollary]{Corollary}
\aliascntresetthe{Corollary}
\newaliascnt{Definition}{Theorem}
\theoremstyle{definition}
\newtheorem{Definition}[Definition]{Definition}
\aliascntresetthe{Definition}

\numberwithin{equation}{section}

\allowdisplaybreaks[1]

\renewcommand{\phi}{\varphi}
\newcommand{\C}{\operatorname{C}}
\newcommand{\N}{\operatorname{N}}
\newcommand{\Z}{\operatorname{Z}}

\newcommand{\Aut}{\operatorname{Aut}}

\newcommand{\Out}{\operatorname{Out}}

\newcommand{\GL}{\operatorname{GL}}
\newcommand{\SL}{\operatorname{SL}}

\newcommand{\Irr}{\operatorname{Irr}}
\newcommand{\IBr}{\operatorname{IBr}}

\newcommand{\Bl}{\operatorname{Bl}}

\newcommand{\tr}{\operatorname{tr}}
\newcommand{\id}{\operatorname{id}}

\mathchardef\ordinarycolon\mathcode`\:  
 \mathcode`\:=\string"8000
 \begingroup \catcode`\:=\active
   \gdef:{\mathrel{\mathop\ordinarycolon}}
 \endgroup

\title{Cartan matrices and Brauer's \\$k(B)$-Conjecture IV}
\author{Benjamin Sambale}
\date{\today}

\begin{document}
\frenchspacing
\maketitle
\begin{abstract}\noindent
In this note we give applications of recent results coming mostly from the third paper of this series. 
It is shown that the number of irreducible characters in a $p$-block of a finite group with abelian defect group $D$ is bounded by $|D|$ (Brauer's $k(B)$-Conjecture) provided $D$ has no large elementary abelian direct summands. 
Moreover, we verify Brauer's $k(B)$-Conjecture for all blocks with minimal non-abelian defect groups. This extends previous results by various authors. 
\end{abstract}

\textbf{Keywords:} blocks, minimal non-abelian defect groups, abelian defect groups, Brauer's $k(B)$-Conjecture\\
\textbf{AMS classification:} 20C15, 20C20

\section{Introduction}
Let $p$ be a prime and let $G$ be a finite group. We consider $p$-blocks $B$ of $G$ with respect to a $p$-modular system which is “large enough” in the usual sense. In two recent articles \cite{SambaleC3,WatanabeAWC} properties of the Cartan matrix $C$ of $B$ have been expressed in terms of the defect group $D$ of $B$. In the present paper we apply these results in order to prove the inequality $k(B)\le|D|$ (Brauer's $k(B)$-Conjecture) in certain cases where $k(B)$ denotes the number of irreducible characters in $B$. Continuing former work by several authors \cite{EKS,GaoZeng,Hendren2,Hendren1,Sambalemna,GaoControl}, we verify Brauer's $k(B)$-Conjecture for all blocks with minimal non-abelian defect groups. Here a group is called \emph{minimal non-abelian} if all its proper subgroups are abelian, but the group itself is non-abelian. This leads also to a proof of Brauer's Conjecture for the $5$-blocks of defect $3$. 

In the last part of the paper we revisit a theorem of Watanabe~\cite{Watanabe1,Watanabe2} about blocks with abelian defect groups. Watanabe has studied a certain correspondence of blocks whenever the inertial group has non-trivial fixed points on $D$ (similar to the $\Z^*$-Theorem). We will show that this correspondence often preserves Cartan matrices up to basic sets (this means up to a transformation of the form $C\mapsto SCS^\text{T}$ for some $S\in\GL(l,\mathbb{Z})$ where $S^\text{T}$ denotes the transpose of $S$). As another tool we show that a coprime action on an abelian $p$-group without elementary abelian direct summands always has a regular orbit. This is used to give a proof of Brauer's $k(B)$-Conjecture for abelian defect groups $D$ such that $D$ has no elementary abelian direct summand of order $p^3$. Improvements of this result for small primes are also presented. In particular, we verify Brauer's Conjecture for $2$-blocks with abelian defect groups of rank at most $7$. This greatly generalizes some results in \cite{Sbrauerfeit}. Some of the proofs rely implicitly on the classification of the finite simple groups.

Most of our notation is standard and can be found in \cite{Feit,Nagao,habil} for example. The number of irreducible Brauer characters of $B$ is denoted by $l(B)$. Moreover, we denote the inertial quotient of $B$ by $I(B)$. Its order $e(B):=|I(B)|$ is the inertial index of $B$. A cyclic group of order $n$ is denoted by $Z_n$, and for convenience, $Z_n^m:=Z_n\times\ldots\times Z_n$ ($n$ copies). Commutators are defined as $[x,y]:=xyx^{-1}y^{-1}$ and groups act from the left as $^ax$. We say that a finite group $A$ acts \emph{freely} on a finite group $H$ if $\C_A(x)=1$ for all $1\ne x\in H$. For an abelian $p$-group $P$ we set $\Omega_i(P):=\{x\in P:x^{p^i}=1\}$ and $\Omega(P):=\Omega_1(P)$.

\section{Fusion systems}
We start by recalling some notation from the theory of fusion systems. Details can be found in \cite{AKO}. Our fusion systems will always be saturated.

\begin{Definition}\label{def}
Let $\mathcal{F}$ be a fusion system on a finite $p$-group $P$. 
\begin{enumerate}[(i)]
\item A subgroup $Q\le P$ is called \emph{fully $\mathcal{F}$-centralized} if $\lvert\C_P(\phi(Q))\rvert\le\lvert\C_P(Q)\rvert$ for all morphisms $\phi:Q\to P$ in $\mathcal{F}$.
\item If $Q$ is fully $\mathcal{F}$-centralized, then there is a fusion system $\C_{\mathcal{F}}(Q)$ on $\C_P(Q)$ defined as follows: a group homomorphism $\phi:R\to S$ \textup{(}$R,S\le\C_P(Q)$\textup{)} belongs to $\C_{\mathcal{F}}(Q)$ if there exists a morphism $\psi:QR\to QS$ in $\mathcal{F}$ such that $\psi_{|Q}=\id_Q$ and $\psi_{|R}=\phi$.
\item If $Q$ is abelian and fully $\mathcal{F}$-centralized, then there is a fusion system $\C_{\mathcal{F}}(Q)/Q$ on $\C_P(Q)/Q$ defined as follows: a group homomorphism $\phi:R/Q\to S/Q$ \textup{(}$Q\le R,S\le\C_P(Q)$\textup{)} belongs to $\C_{\mathcal{F}}(Q)/Q$ if there exists a morphism $\psi:R\to S$ in $\C_{\mathcal{F}}(Q)$ such that $\psi(u)Q=\phi(uQ)$ for all $u\in R$.
\end{enumerate}
\end{Definition}

If in the situation of \autoref{def} the group $Q$ is cyclic, say $Q=\langle u\rangle$, then we write $\C_{\mathcal{F}}(u)$ instead of $\C_{\mathcal{F}}(\langle u\rangle)$.

Let $B$ be a block of a finite group $G$ with defect group $D$. Recall that a ($B$-)\emph{subsection} is a pair $(u,b_u)$ such that $u\in D$ and $b_u$ is a Brauer correspondent of $B$ in $\C_G(u)$. If $b_u$ and $B$ have the same defect, the subsection is called \emph{major}. This holds for example for the \emph{trivial} subsection $(1,B)$. More generally, a ($B$-)\emph{subpair} is a pair $(Q,b_Q)$ such that $Q\le D$ and $b_Q$ is a Brauer correspondent of $B$ in $\C_G(Q)$. In case $Q=D$, we say $(D,b_D)$ is a Sylow $B$-subpair.
It is well-known that every block $B$ of a finite group with defect group $D$ determines a fusion system $\mathcal{F}$ on $D$ which describes the conjugation of subpairs. In this setting, $I(B)\cong\Out_{\mathcal{F}}(D)$. By the Schur-Zassenhaus Theorem we can consider $I(B)$ as a subgroup of $\Aut(D)$. 

The next lemma might be already known, but we were unable to find a reference (cf. \cite[Theorem~1.5]{Olsson}). Therefore a proof is given.

\begin{Lemma}\label{fusiondom}
Let $B$ be a $p$-block of a finite group $G$ with defect group $D$ and fusion system $\mathcal{F}$. Let $Z\le\Z(G)$ be a $p$-subgroup. Then $B$ dominates a unique block $\overline{B}$ of $G/Z$ with defect group $D/Z$ and fusion system $\mathcal{F}/Z$.
\end{Lemma}
\begin{proof}
Since $Z\unlhd G$, we have $Z\le D$. Moreover, it is easy to see that $\mathcal{F}=\C_{\mathcal{F}}(Z)$. Hence, $\mathcal{F}/Z$ is well defined. The uniqueness of $\overline{B}$ and its defect group can be found in \cite[Theorem~5.8.11]{Nagao}. 
It remains to determine the fusion system of $\overline{B}$.
For $H\le G$ we write $\overline{H}:=HZ/Z$. We fix a Sylow $B$-subpair $(D,b_D)$. For every subgroup $Z\le Q\le D$ there exists a unique $B$-subpair $(Q,b_Q)$ such that $(Q,b_Q)\le(D,b_D)$. Let $\C_{\overline{G}}(\overline{Q})=\overline{C_Q}$ with $\C_G(Q)\le C_Q\le \N_G(Q)$. Moreover, let $\beta_Q:=b_Q^{C_Q}$, and let $\overline{\beta_Q}$ be the unique block of $\C_{\overline{G}}(\overline{Q})$ dominated by 
$\beta_Q$.
We claim that $(\overline{Q},\overline{\beta_Q})$ is a $\overline{B}$-subpair. To prove this, we need to show that $\overline{\beta_Q}^{\overline{G}}=\overline{B}$. Let $e_B$ be the block idempotent of $B$ with respect to an algebraically closed field $F$ of characteristic $p$. Let $\theta:FG\to F\overline{G}$ be the canonical epimorphism. Then $\theta(e_B)=e_{\overline{B}}$. Let $\omega_{\beta_Q}$ be the central character of $\beta_Q$. Then, by \cite[Lemma~5.8.5]{Nagao}, the central character $\omega_{\overline{\beta_Q}}$ of $\overline{\beta_Q}$ satisfies $\omega_{\beta_Q}=\omega_{\overline{\beta_Q}}\circ\theta$ where $\theta$ is identified with its restriction to $\Z(FC_Q)$. Let 
\[\eta:\Z(FG)\to\Z(FC_Q),\ \sum_{g\in G}{\alpha_gg}\mapsto\sum_{g\in C_Q}{\alpha_gg}\hspace{2cm}(\alpha_g\in F).\]
Then the analogous map $\overline{\eta}:\Z(F\overline{G})\to\Z(F\C_{\overline{G}}(\overline{Q}))$ is the Brauer homomorphism.
Moreover,
\[\omega_{\overline{\beta_Q}}(\overline{\eta}(e_{\overline{B}}))=\omega_{\overline{\beta_Q}}(\overline{\eta}(\theta(e_B)))=\omega_{\overline{\beta_Q}}(\theta(\eta(e_B)))=\omega_{\beta_Q}(\eta(e_B))=\omega_B(e_B)=1.\]

This shows that $\overline{\beta_Q}^{\overline{G}}=\overline{B}$ and $(\overline{Q},\overline{\beta_Q})$ is a $\overline{B}$-subpair. 
In particular, $(\overline{D},\overline{\beta_D})$ is a Sylow $\overline{B}$-subpair. Suppose that $(R,b_R)\unlhd (S,b_S)$ for some subgroups $Z\le R\unlhd S\le D$. Then $b_R^{\C_G(R)S}=b_S^{\C_G(R)S}$. 
As we have seen above, \[\overline{\beta_R}^{\C_{\overline{G}}(\overline{R})\overline{S}}=\overline{\beta_R}^{\overline{C_RS}}=\overline{\beta_R^{C_RS}}=\overline{b_R^{C_RS}}=\overline{b_S^{C_RS}}=\overline{\beta_S^{C_RS}}=\overline{\beta_S}^{\overline{C_RS}}=\overline{\beta_S}^{\C_{\overline{G}}(\overline{R})\overline{S}}\]
(observe that $\C_G(R)S\le C_RS\le G$). 
This implies $(\overline{R},\overline{\beta_R})\unlhd(\overline{S},\overline{\beta_S})$. Therefore the poset of $B$-subpairs $(Q,b_Q)\le(D,b_D)$ such that $Z\le Q$ is in one-to-one correspondence with the poset of $\overline{B}$-subpairs via Brauer correspondence and $\theta$. Let $\mathcal{F}'$ be the fusion system of $\overline{B}$. Suppose that $\overline{\phi}:\overline{R}\to\overline{S}$ is a morphism in $\mathcal{F}'$ for $Z\le R,S\le D$. Then there exists a $g\in G$ such that 
$\overline{g}(\overline{R},\overline{\beta_R})\overline{g}^{-1}\le(\overline{S},\overline{\beta_S})$ and $\overline{\phi}(\overline{x})=\overline{g}\overline{x}\overline{g}^{-1}$ for all $\overline{x}\in\overline{R}$. 
Obviously, we have $gRg^{-1}\le S$. Moreover, $\overline{g\beta_Rg^{-1}}=\overline{g}\overline{\beta_R}\overline{g}^{-1}=\overline{\beta_{gRg^{-1}}}$ and \[(gb_Rg^{-1})^{C_{gRg^{-1}}}=g(b_R^{C_R})g^{-1}=g\beta_Rg^{-1}=\beta_{gRg^{-1}}=b_{gRg^{-1}}^{C_{gRg^{-1}}}.\]
It follows that there exists an element $h\in C_{gRg^{-1}}\le\N_G(gRg^{-1})$ such that $hgb_Rg^{-1}h^{-1}=b_{gRg^{-1}}$ and $\overline{\phi}(\overline{x})=\overline{hgxg^{-1}h^{-1}}$ for $\overline{x}\in\overline{R}$. Therefore, $hg(R,b_R)g^{-1}h^{-1}\le(S,b_S)$ and the map $\phi:R\to S$ such that $\phi(x):=hgxg^{-1}h^{-1}$ for $x\in R$ is a morphism in $\mathcal{F}$. Conversely, if $\phi:R\to S$ is given in $\mathcal{F}$, then it is easy to see that the corresponding map $\overline{\phi}$ lies in $\mathcal{F}'$. Consequently, $\mathcal{F}'=\mathcal{F}/Z$. 
\end{proof}

\begin{Lemma}\label{fusion}
Let $B$ be a block of a finite group $G$ with defect group $D$ and fusion system $\mathcal{F}$. Let $(u,b)$ be a $B$-subsection such that $\langle u\rangle$ is fully $\mathcal{F}$-centralized. Then $b$ has defect group $\C_D(u)$ and fusion system $\C_{\mathcal{F}}(u)$. Moreover, $b$ dominates a unique block $\overline{b}$ of $\C_G(u)/\langle u\rangle$ with defect group $\C_D(u)/\langle u\rangle$ and fusion system $\C_{\mathcal{F}}(u)/\langle u\rangle$. In particular, we have canonical isomorphisms 
\[I(\overline{b})\cong I(b)\cong\C_{\Out_{\mathcal{F}}(\C_D(u))}(u).\] 
If $\overline{b}$ has Cartan matrix $\overline{C}$, then $b$ has Cartan matrix $|\langle u\rangle|\overline{C}$. In particular, $l(b)=l(\overline{b})$.
\end{Lemma}
\begin{proof}
The first claim follows from \cite[Theorem~IV.3.19]{AKO}. The uniqueness of $\overline{b}$ and the claim about the Cartan matrices can be found in \cite[Theorem~5.8.11]{Nagao}. The fusion system of $\overline{b}$ was determined in \autoref{fusiondom}.
It is well-known that the inertial quotient $I(b)\cong\Out_{\C_{\mathcal{F}}(u)}(\C_D(u))$ is a $p'$-group. Thus, \cite[Theorem~6.3(i)]{Linckelmann} implies $\Out_{\C_{\mathcal{F}}(u)}(\C_D(u))\cong\Out_{\C_{\mathcal{F}}(u)/\langle u\rangle}(\C_D(u)/\langle u\rangle)\cong I(\overline{b})$. Finally, the isomorphism $I(b)\cong\C_{\Out_{\mathcal{F}}(\C_D(u))}(u)$ follows from the definition of $\C_\mathcal{F}(u)$.
\end{proof}

We also recall two important subgroups related to fusion systems.

\begin{Definition}
Let $\mathcal{F}$ be a fusion system on a finite $p$-group $P$. 
\begin{enumerate}[(i)]
\item $\mathfrak{foc}(\mathcal{F}):=\langle f(x)x^{-1}: x\in Q\le P,\ f\in\Aut_{\mathcal{F}}(Q)\rangle$ is called the \emph{focal subgroup} of $\mathcal{F}$.
\item $\Z(\mathcal{F}):=\{x\in P: x\text{ is fixed by every morphism in }\mathcal{F}\}$ is called the \emph{center} of $\mathcal{F}$.
\end{enumerate}
\end{Definition}

If $B$ is a block with fusion system $\mathcal{F}$ and defect group $D$, then we set $\mathfrak{foc}(B):=\mathfrak{foc}(\mathcal{F})$ (but $\Z(B)$ is usually used for the center of the block algebra). 
We say that $B$ is \emph{controlled} if all morphisms of $\mathcal{F}$ are generated by restrictions from $\Aut_{\mathcal{F}}(D)$. In this case, $\mathfrak{foc}(B)=[D,I(B)]$ and $\Z(\mathcal{F})=\C_D(I(B))$. If $D$ is abelian, then $B$ is controlled and 
$D=[D,I(B)]\oplus\C_D(I(B))$ (see \cite[Theorem~2.3]{Gorenstein}).

\section{Non-abelian defect groups}


\begin{Theorem}\label{main}
Let $B$ be a $p$-block of a finite group with non-abelian defect group $D$. 
Suppose that $D/\langle z\rangle$ is abelian of rank $2$ for some $z\in\Z(D)$.
Then $k(B)\le|D|$.
\end{Theorem}
\begin{proof}
Let $x,y\in D$ such that $D=\langle x,y,z\rangle$. Since $D$ is non-abelian, $1\ne [x,y]\in D'\subseteq\langle z\rangle$. Let $\alpha\in\C_{\Aut(D)}(z)$ be a $p'$-automorphism. We write $\alpha(x)\equiv x^iy^j\pmod{\langle z\rangle}$ and $\alpha(y)\equiv x^ky^l\pmod{\langle z\rangle}$ with $i,j,k,l\in\mathbb{Z}$. By \cite[III.1.2, III.1.3]{Huppert}, 
\[[x,y]=\alpha([x,y])=[x^iy^j,x^ky^l]=[x,y]^{il-jk}\]
and therefore $il-jk\equiv 1\pmod{p}$. Hence, $\alpha$ corresponds to a matrix with determinant $1$ under the isomorphism $\Aut(D/\langle x^p,y^p,z\rangle)\cong\Aut(Z_p^2)\cong\GL(2,p)$. If $x$ and $y$ have the same order modulo $\langle z\rangle$, then $\alpha$ also corresponds to a matrix with determinant $1$ under the isomorphism $\Aut(\Omega(D/\langle z\rangle))\cong\GL(2,p)$. Now assume, without loss of generality, that $x$ has larger order than $y$ modulo $\langle z\rangle$. Then $p\mid k$, since $\alpha(y)$ and $y$ have the same order. In particular $il\equiv 1\pmod{p}$. Let $p^n$ be the order of $x$ modulo $\langle z\rangle$. Then obviously, $\alpha(x^{p^{n-1}})\equiv x^{ip^{n-1}}\pmod{\langle z\rangle}$. This show that $\alpha$ induces an upper triangular matrix with determinant $1$ in $\Aut(\Omega(D/\langle z\rangle))$. Hence, in any case $\alpha$ corresponds to an element of $\SL(\Omega(D/\langle z\rangle))$. 

Now suppose that $\alpha$ has a non-trivial fixed point in $D/\langle z\rangle$. Then there is also a non-trivial fixed point in $\Omega(D/\langle z\rangle)$. It follows that $\alpha$ is conjugate to a unitriangular matrix under $\Aut(\Omega(D/\langle z\rangle))\cong\GL(2,p)$. However, then $\alpha$ acts trivially on $\Omega(D/\langle z\rangle)$, since $\alpha$ is a $p'$-element. By \cite[Theorem~5.2.4]{Gorenstein}, $\alpha$ also acts trivially on $D/\langle z\rangle$. This forces $\alpha=1$ by \cite[Theorem~5.3.2]{Gorenstein}. Therefore we have shown that every $p'$-automorphism of $\C_{\Aut(D)}(z)$ acts freely on $D/\langle z\rangle$.

Now let $\mathcal{F}$ be the fusion system of $B$.
Let $(z,b_z)$ be a (major) subsection of $B$. Since $z\in\Z(D)$, the subgroup $\langle z\rangle$ is fully $\mathcal{F}$-centralized. By \autoref{fusion}, $b_z$ dominates a block $\overline{b_z}$ of $G/\C_G(z)$ with abelian defect group $\overline{D}:=D/\langle z\rangle$ and inertial quotient $I(\overline{b_z})\cong\C_{I(B)}(z)$. 
As we have seen above, $I(\overline{b_z})$ acts freely on $\overline{D}$. In particular, all non-trivial $\overline{b_z}$-subsections $(u,\beta_u)$ have inertial index $1$. This implies $l(\beta_u)=1$, since $\overline{D}$ is abelian (see \cite[Theorem~V.9.13]{Feit}). Let $\overline{C}$ be the Cartan matrix of $\overline{b_z}$.
Then we deduce from a result of Fujii~\cite[Corollary~1]{DetCartan} that $\det\overline{C}=|\overline{D}|$. Since $|\langle z\rangle|\overline{C}$ is the Cartan matrix of $b_z$, the claim follows from \cite[Theorem~11]{SambaleC3}. 
\end{proof}

\begin{Corollary}\label{mna}
Brauer's $k(B)$-Conjecture holds for all blocks with minimal non-abelian defect groups.
\end{Corollary}
\begin{proof}
The minimal non-abelian $p$-groups were classified by Rédei (see \cite[Aufgabe~III.7.22]{Huppert}), but the present proof can go without a detailed structure knowledge. 
Let $D$ be a minimal non-abelian defect group of a block $B$. Then there are non-commuting elements $x,y\in D$. Since $\langle x,y\rangle$ is non-abelian, we have $D=\langle x,y\rangle$. Now let $u\in\Phi(D)$ and $v\in D$ be arbitrary. Then $v$ lies in a maximal subgroup $M<D$ and so does $u$. Since $M$ is abelian, it follows that $[u,v]=1$. This shows that $\Phi(D)\subseteq\Z(D)$. In particular $z:=[x,y]\in D'\subseteq\Phi(D)\subseteq\Z(D)$. Since $D/\langle z\rangle$ is abelian of rank $2$, the claim follows from \autoref{main}.
\end{proof}

\autoref{mna} includes the non-abelian defect groups of order $p^3$. In particular, this extends results by Hendren~\cite[Theorem~4.10]{Hendren1}. Apart from minimal non-abelian groups, \autoref{main} also applies to other groups like the central product $D_8\mathop{\ast}Z_{2^n}$ for some $n\ge 2$ where $D_8$ is the dihedral group of order $8$.

In \cite[Corollary~1]{SambalekB2} we have proved that Brauer's $k(B)$-Conjecture holds for the $3$-blocks of defect $3$. Now we can do the same for $p=5$.

\begin{Corollary}
Brauer's $k(B)$-Conjecture holds for the $5$-blocks of defect at most $3$.
\end{Corollary}
\begin{proof}
The abelian defect groups of order at most $5^3$ have been handled in \cite[Theorem~14.17]{habil} (see also \autoref{p35} below). In the non-abelian case, \autoref{mna} applies.
\end{proof}

Out next results concerns a larger class of $p$-groups, but introduces restrictions on $p$. The proof makes use of a recent result by Watanabe~\cite{WatanabeAWC}.

\begin{Theorem}
Let $p\le 5$, and let $B$ be a $p$-block of a finite group with defect group $D$. Suppose that $D/\langle z\rangle$ is metacyclic for some $z\in\Z(D)$. Then $k(B)\le|D|$.
\end{Theorem}
\begin{proof}
The case $p=2$ is already known (see \cite[Theorem~13.8]{habil}). Thus, let $p\in\{3,5\}$.
If $D$ is abelian, then the rank of $D$ is at most $3$ and the result follows from \cite[Theorems~14.16 and 14.17]{habil}. Now assume that $D$ is non-abelian. If $D/\langle z\rangle$ is abelian, then \autoref{main} applies.
Thus, we may assume that $D/\langle z\rangle$ is non-abelian. If $p=3$, then the claim follows from \cite[Proposition~8.16]{habil}. Therefore, let $p=5$. Let $(z,b_z)$ be a $B$-subsection. As before, $b_z$ dominates a block $\overline{b_z}$ with non-abelian, metacyclic defect group $D/\langle z\rangle$. 
By a result of Stancu~\cite{Stancu} the fusion system $\overline{\mathcal{F}_z}$ of $\overline{b_z}$ is controlled. 
Moreover, the possible automorphism groups $I(\overline{b_z})$ are described in a paper by Sasaki~\cite{Sasaki}. It follows that $\mathfrak{foc}(\overline{b_z})=[D/\langle z\rangle,I(\overline{b_z})]$ is cyclic (for details see \cite[proof of Theorem~8.8]{habil}). 
Hence, by the main result of \cite{WatanabeAWC}, $l(b_z)=l(\overline{b_z})\mid 4$. In case $l(b_z)\le 2$, the claim follows from \cite[Theorem~4.9]{habil}. Finally, let $l(b_z)=4$. Let $|\langle z\rangle|=5^n$, and let $C$ be the Cartan matrix of $b_z$. 
By \cite[Corollary on p.181]{WatanabeAWC}, $C$ has elementary divisors $5^a$ and $|D|$ where $|D|$ occurs with multiplicity $1$ and $a\ge n$. 
Choose a basic set such that $C$ has block form
\[C=\begin{pmatrix}
C_1&0\\0&C_2
\end{pmatrix}\]
where $C_1\in\mathbb{Z}^{r\times r}$ does not split further (for any basic set) and $r\le4$ (possibly $r=4$). Without loss of generality, $|D|$ is an elementary divisor of $C_1$. 
By way of contradiction, we may assume that there is a vector $0\ne x\in\mathbb{Z}^4$ such that $x|D|C^{-1}x^{\text{T}}<4$ (see \cite[Theorem~V.9.17]{Feit}). Looking into the proof of \cite[Theorem~V.9.17]{Feit} more closely, reveals that there is a character $\chi\in\Irr(B)$ such that the row of generalized decomposition numbers $d_\chi:=(d^u_{\chi\phi}:\phi\in\IBr(b_z))$ satisfies 
\[\tr\bigl(d_\chi|D|C^{-1}\overline{d_\chi}^{\text{T}}\bigr)<4[\mathbb{Q}(\zeta):\mathbb{Q}]=16\cdot 5^{n-1}\] 
where $\zeta$ is a primitive $5^n$-th root of unity and $\tr$ is the trace of the Galois extension $\mathbb{Q}(\zeta)|\mathbb{Q}$. We may write $d_\chi=(d_1,d_2)$ where $d_1\in\mathbb{C}^r$ and $d_2\in\mathbb{C}^{4-r}$. Then \[\tr\bigl(d_\chi|D|C^{-1}\overline{d_\chi}^{\text{T}}\bigr)=\tr\bigl(d_1|D|C_1^{-1}\overline{d_1}^{\text{T}}\bigr)+\tr\bigl(d_2|D|C_2^{-1}\overline{d_2}^{\text{T}}\bigr).\]
Since all entries of $|D|C_2^{-1}$ are divisible by $5$, it follows that $\tr\bigl(d_2|D|C_2^{-1}\overline{d_2}^{\text{T}}\bigr)\ge 5\phi(5^n)=20\cdot 5^{n-1}$ or $\tr\bigl(d_2|D|C_2^{-1}\overline{d_2}^{\text{T}}\bigr)=0$. The first case is impossible. Hence, $d_2=0\in\mathbb{Z}^{4-r}$. Since $d_\chi$ consists of algebraic integers, we may write
\[d_\chi=\sum_{i=0}^{\phi(5^n)-1}{a_i\zeta^i}\]
for some $a_i\in\mathbb{Z}^4$. Let us write $Q(x,y):=x|D|C^{-1}\overline{y}^\text{T}$ for $x,y\in\mathbb{C}^4$. Then $Q$ is a positive definite Hermitian form. 
Moreover,
\[\alpha:=Q(d_\chi,d_\chi)=a_0^*+\sum_{i=1}^{2\cdot 5^{n-1}-1}{a_i^*(\zeta^i+\zeta^{-i})}\]
for some $a_i^*\in\mathbb{Z}$. Since $\zeta^{2\cdot 5^{n-1}}+\zeta^{-2\cdot 5^{n-1}}=-1-\zeta^{5^{n-1}}-\zeta^{-5^{n-1}}$, we get
\[a_0^*=\sum_{i=0}^{\phi(5^n)-1}{Q(a_i,a_i)}-\sum_{\substack{0\le s<t<\phi(5^n),\\t-s\equiv \pm 2\cdot5^{n-1}\pmod{5^n}}}{Q(a_s,a_t)}>0\]
and
\[a_{5^{n-1}}^*=\sum_{\substack{0\le s<t<\phi(5^n),\\t-s\equiv 5^{n-1}\pmod{5^n}}}{Q(a_s,a_t)}-\sum_{\substack{0\le s<t<\phi(5^n),\\t-s\equiv \pm 2\cdot5^{n-1}\pmod{5^n}}}{Q(a_s,a_t)}.\]

Suppose for the moment that $\chi$ has positive height. Then the $5$-adic valuation of $\alpha$ is strictly larger than $1$ (see \cite[Proposition~1.36]{habil}). In particular, $\alpha/5$ is an algebraic integer (this can be seen by going over to the cyclotomic field over the $5$-adic numbers, see \cite[Proposition~II.7.13]{NeukirchE}). Since $1$, $\zeta+\zeta^{-1},\ldots,\zeta^{2\cdot 5^{n-1}-1}+\zeta^{-2\cdot 5^{n-1}+1}$ is a basis for the ring of real algebraic integers, we have $5\mid a_i^*$ for all $i$. Moreover,
\[\tr(\alpha)=a_0^*\phi(5^n)+\sum_{i=1}^{2\cdot 5^{n-1}-1}{a_i^*\tr(\zeta^i+\zeta^{-i})}=a_0^*\phi(5^n)-2\cdot5^{n-1}a_{5^{n-1}}^*.\]
If $a_{5^{n-1}}^*\le0$, then we obtain the contradiction $\tr(\alpha)\ge a_0^*\phi(5^n)\ge 20\cdot 5^{n-1}$. Thus, $a_{5^{n-1}}^*>0$. Observe that
\begin{align*}
a_{5^{n-1}}^*&\le \frac{1}{2}\sum_{i=0}^{5^{n-1}-1}{Q(a_i,a_i)}+\sum_{i=5^{n-1}}^{3\cdot5^{n-1}-1}{Q(a_i,a_i)}+\frac{1}{2}\sum_{i=3\cdot 5^{n-1}}^{4\cdot 5^{n-1}-1}{Q(a_i,a_i)}-\sum_{\substack{0\le s<t<\phi(5^n),\\t-s\equiv \pm 2\cdot5^{n-1}\pmod{5^n}}}{Q(a_s,a_t)}\\ 
&=a_0^*-\frac{1}{2}\sum_{i=0}^{5^{n-1}-1}{Q(a_i,a_i)}-\frac{1}{2}\sum_{i=3\cdot 5^{n-1}}^{4\cdot 5^{n-1}-1}{Q(a_i,a_i)}.
\end{align*}
Now it is easy to see that $a_0^*>a_{5^{n-1}}^*$ and thus $a_0^*\ge a_{5^{n-1}}^*+5$. This gives the contradiction $\tr(\alpha)\ge 20\cdot 5^{n-1}+2\cdot 5^{n-1}a_{5^{n-1}}^*\ge 20\cdot 5^{n-1}$. Therefore, we have shown that $\chi$ has height $0$. 

In particular, $d_\chi|D|C^{-1}\overline{d_\psi}^\text{T}\ne 0$ for all $\psi\in\Irr(B)$ (see \cite[Proposition~1.36]{habil}). Since $d_2=0$, it follows that the first $r$ components of $d_\psi$ cannot all be zero. Hence, in order to bound $k(B)$ by the number of rows $d_\psi$, we may work with the matrix $C_1$ instead of $C$. This means it suffices to show
\[\min\{x|D|C_1^{-1}x^\text{T}:0\ne x\in\mathbb{Z}^r\}\ge r\]
(cf. \cite[Prop.~2.2]{Plesken}).

The integral matrix $\overline{C_1}:=5^{-a}C_1$ has elementary divisors $1$ and $5^{-a}|D|$ where $5^{-a}|D|$ occurs with multiplicity $1$. In particular, $\det\overline{C_1}=5^{-a}|D|$.
Since $r\le 4$, it is known that $\overline{C_1}$ can be factorized in the form
\[\overline{C_1}=Q_1^{\text{T}}Q_1\]
where $Q_1\in\mathbb{Z}^{k\times r}$ for some $k\in\mathbb{N}$ (see \cite{Mordell}). We may assume that $Q_1$ has no vanishing rows. 
By the choice of $C_1$, the matrix $Q_1$ is indecomposable with the notation of \cite[Definition~1]{SambaleC3}. Now \cite[Lemma~4]{SambaleC3} implies 
\[\min\{x|D|C_1^{-1}x^\text{T}:0\ne x\in\mathbb{Z}^r\}=\min\{\det(\overline{C_1})x\overline{C_1}^{-1}x^\text{T}:0\ne x\in\mathbb{Z}^r\}\ge r.\]
This completes the proof.
\end{proof}

Most parts of the proof above also work for any odd prime $p$. However, the splitting theorem by Mordell~\cite{Mordell} is no longer true for matrices of larger dimension. Consider for example the following situation: $p=7$, $z=1$, $l(B)=6$ and 
\[C=C_1=7^2\begin{pmatrix}
3&.&1&.&.&.\\
.&2&.&1&.&.\\
1&.&2&1&.&.\\
.&1&1&2&1&.\\
.&.&.&1&2&1\\
.&.&.&.&1&2
\end{pmatrix}
\]
(the matrix is a modified version of the $E_6$ lattice). Then $\det(7^{-2}C)=7$ and there is no factorization of the form $7^{-2}C=Q^\text{T}Q$ for some integral matrix $Q$. In fact 
\[\min\{x7^3C^{-1}x^\text{T}:0\ne x\in\mathbb{Z}^6\}=4<6.\]
However, we do not know if $C$ can actually occur as a Cartan matrix of a block.


\section{Abelian defect groups}

We begin with a remark about a theorem of Watanabe~\cite{Watanabe1}.

\begin{Lemma}\label{kBdom}
Let $B$ be a $p$-block of a finite group $G$ with abelian defect group, and let $Z$ be a central $p$-subgroup of $G$. Then $k(B)=|Z|k(\overline{B})$ where $\overline{B}$ is the unique block of $G/Z$ dominated by $B$.
\end{Lemma}
\begin{proof}
Let $D$ be a defect group of $B$. Obviously, $Z\subseteq\C_D(I(B))$. 
Let $\mathcal{R}$ be a set of representatives for the $I(B)$-conjugacy classes of $[D,I(B)]$. 
Then $\{(uz,b_{uz}):u\in\mathcal{R},\ z\in\C_D(I(B))\}$ is a set of representatives of the $G$-conjugacy classes of $B$-subsections. By \cite[Corollary~1]{Watanabe1}, we have $l(b_{uz})=l(b_u)$ for all $z\in\C_D(I(B))$. 
This shows
\[k(B)=\sum_{u\in\mathcal{R}}\sum_{z\in\C_D(I(B))}{l(b_{uz})}=\lvert\C_D(I(B))\rvert\sum_{u\in\mathcal{R}}{l(b_u)}.\]
Now we consider the block $\overline{B}$. 
For $H\le G$ and $x\in D$ we write $\overline{H}:=HZ/Z$ and $\overline{x}:=xZ$.
Let $\C_{\overline{G}}(\overline{x})=\overline{C_x}$ with $\C_G(\langle x\rangle Z)=\C_G(x)\le C_x\le \N_G(\langle x\rangle Z)$. 
Moreover, let $\overline{b_x^{C_x}}$ be the unique block of $\C_{\overline{G}}(\overline{x})$ dominated by $b_x^{C_x}$. 
Choose a transversal $\mathcal{S}\subseteq G$ for the cosets $\overline{\C_D(I(B))}$.
Since $I(B)\cong I(\overline{B})$, the set
\[\{(\overline{uz},\overline{b_{uz}^{C_{uz}}}):u\in\mathcal{R},\ z\in\mathcal{S}\}\]
represents the $\overline{B}$-subsections up to $\overline{G}$-conjugacy (cf. proof of \autoref{fusiondom}). By \cite[Theorem~5.8.11]{Nagao}, $l(\overline{b_{uz}^{C_{uz}}})=l(b_{uz}^{C_{uz}})$. Since $C_{uz}$ acts trivially on $\langle \overline{uz}\rangle$ and on $Z$, it follows that $C_{uz}/\C_G(uz)$ is a $p$-group. 
From the properties of fusion systems it is clear that $\N_G(\langle uz\rangle Z,b_{uz})/\C_G(uz)$ is a $p'$-group. Hence, $\N_G(\langle uz\rangle Z,b_{uz})\cap C_{uz}=\C_G(uz)$ and the Fong-Reynolds Theorem implies $l(b_{uz}^{C_{uz}})=l(b_{uz})=l(b_u)$.
Consequently,
\[k(\overline{B})=\sum_{u\in\mathcal{R}}\sum_{z\in\mathcal{S}}{l(\overline{b_{uz}^{C_{uz}}})}=\lvert\overline{\C_D(I(B))}\rvert\sum_{u\in\mathcal{R}}{l(b_u)}.\]
This proves the claim.
\end{proof}

The statement of \autoref{kBdom} is not true for non-abelian defect groups, as it can be seen from the principal $2$-block of $\SL(2,3)$ with $Z:=\Z(\SL(2,3))$. 

Next, we need a result about the so-called $*$-construction introduced in \cite{BrouePuig}. 

\begin{Lemma}\label{astconst}
Let $B$ be a $p$-block of a finite group with defect group $D$. Let $u\in D$ and let $(u,b)$ be a $B$-subsection. Let $\chi\in\Irr(B)$, $\phi\in\IBr(b)$, and let $\lambda\in\Irr(D/\mathfrak{foc}(B))\subseteq\Irr(D)$. Then $\lambda\mathop{\ast}\chi\in\Irr(B)$ and
\[d^u_{\lambda\mathop{\ast}\chi,\phi}=\lambda(u)d^u_{\chi\phi}.\]
\end{Lemma}
\begin{proof}
We use the approach from \cite[Section~1]{RobinsonFocal}. Our first claim is already proved there.
Let $\mathcal{R}$ be a set of representatives for the $G$-conjugacy classes of $B$-subsections such that $(u,b)\in\mathcal{R}$. For $(v,b_v)\in\mathcal{R}$, $\psi\in\IBr(b_v)$ and $x\in\C_G(v)$ let
\[\widetilde{\psi}(x):=\begin{cases}\psi(s)&\text{if }x=vs\text{ where }s\in\C_G(v)_{p'},\\0&\text{otherwise}\end{cases}\]
where $\C_G(v)_{p'}$ denotes the set of $p$-regular elements of $\C_G(v)$. 
Then $\widetilde{\psi}$ is a class function on $\C_G(v)$, and it is well-known (as a consequence of Brauer's Second Main Theorem) that
\[\chi=\sum_{(v,b_v)\in\mathcal{R}}\sum_{\psi\in\IBr(b_v)}{d^v_{\chi\psi}\widetilde{\psi}^G}.\]
By \cite{RobinsonFocal} we have
\[\lambda\mathop{\ast}\chi=\sum_{(v,b_v)\in\mathcal{R}}\sum_{\psi\in\IBr(b_v)}{\lambda(v)d^v_{\chi\psi}\widetilde{\psi}^G}.\]
Therefore, it suffices to show that the functions $\{\widetilde{\psi}^G:(v,b_v)\in\mathcal{R},\psi\in\IBr(b_v)\}$ are linearly independent over $\mathbb{C}$. Thus, assume that
\[\Phi:=\sum_{(v,b_v)\in\mathcal{R}}\sum_{\psi\in\IBr(b_v)}{\alpha_{\psi}\widetilde{\psi}^G}=0\]
for some $\alpha_{\psi}\in\mathbb{C}$. Let $(v,b_v),(v',b_{v'})\in\mathcal{R}$ such that $v$ and $v'$ are not conjugate in $G$. Then the functions $\widetilde{\psi}^G$ and $\widetilde{\psi'}^G$ for $\psi\in\IBr(b_v)$ and $\psi'\in\IBr(b_{v'})$ have disjoint support. Hence, it suffices to consider partial sums of $\Phi$ corresponding to subsets $\mathcal{S}$ of the form 
\[\mathcal{S}:=\{(v,b_v)\in\mathcal{R}:v\text{ is conjugate to $u$ in }G\}.\]
Choose $1=x_1,\ldots,x_n\in G$ such that $\mathcal{S}=\{(x_iux_i^{-1},b_{x_iux_i^{-1}}):i=1,\ldots,n\}$. Then $\{x_i^{-1}b_{x_iux_i^{-1}}x_i:i=1,\ldots,n\}$ is the set of Brauer correspondents of $B$ in $\C_G(u)$. Moreover, for $s\in\C_G(u)_{p'}$ we have
\begin{align*}
\Phi(us)&=\sum_{(v,b_v)\in\mathcal{S}}\sum_{\psi\in\IBr(b_v)}{\alpha_{\psi}\widetilde{\psi}^G(us)}=\sum_{i=1}^n\sum_{\psi\in\IBr(b_{x_iux_i^{-1}})}{\alpha_{\psi}\bigl({^{x_i^{-1}}\psi}\bigr)(s)}\\
&=\sum_{\substack{b\in\Bl(\C_G(u)),\\b^G=B}}\sum_{\psi\in\IBr(b)}{\alpha^*_{\psi}\psi(s)}
\end{align*}
where $\alpha^*_{\psi}:=\alpha_{\psi'}$ if $^{x_i}\psi=\psi'$ for some $i\in\{1,\ldots,n\}$. Since the irreducible Brauer characters of $\C_G(u)$ are linearly independent as functions on $\C_G(u)_{p'}$ (see \cite[Lemma~IV.3.4]{Feit}), the claim follows.
\end{proof}

The following result generalizes \cite[Corollary~13]{SambaleC3}.

\begin{Proposition}\label{freeact}
Let $B$ be a block of a finite group with abelian defect group $D$. Suppose that there is an element $u\in D$ such that $\C_{I(B)}(u)$ acts freely on $[D,\C_{I(B)}(u)]$. Then $k(B)\le|D|$. This applies in particular, if $[D,\C_{I(B)}(u)]$ is cyclic or if $\C_{I(B)}(u)$ has prime order.
\end{Proposition}
\begin{proof}
Let $(u,b)$ be a $B$-subsection. We will determine the shape of the Cartan matrix $C_u$ of $b$. By \autoref{fusion}, $b$ has defect group $D$ and inertial quotient $I(b)\cong\C_{I(B)}(u)$. Let $Z:=\C_{D}(I(b))$, and let $b_Z$ be a Brauer correspondent of $b$ in $\C_G(Z)\,(\subseteq\C_G(u))$. By \cite[Corollary]{Watanabe2} (applied repeatedly), the elementary divisors of the Cartan matrices of $b$ and $b_Z$ coincide (counting multiplicities). Let $\overline{b_Z}$ the block of $\C_G(Z)/Z$ dominated by $b_Z$ with defect group $\overline{D}:=D/Z$. Then $I(\overline{b_Z})\cong I(b)$ acts freely on $\overline{D}\cong[D,\C_{I(B)}(u)]$. Hence, a result by Fujii~\cite{DetCartan} implies that the elementary divisors of the Cartan matrix of $\overline{b_Z}$ are $1$ and $|\overline{D}|$ where $|\overline{D}|$ occurs with multiplicity $1$. Consequently, the elementary divisors of $C_u$ are $|Z|$ and $|D|$ where $|D|$ occurs with multiplicity $1$. In particular, $\widetilde{C}_u:=|Z|^{-1}C_u$ is an integral matrix with determinant $|\overline{D}|$. Let $Q_u$ be the decomposition matrix of $b$. 
By the proof of \cite[Theorem~2]{RobinsonFocal} we have $\lambda\mathop{\ast}\chi\ne\chi$ for every $\chi\in\Irr(b)$ and $1\ne\lambda\in\Irr(D/[D,I(b)])\cong\Irr(Z)$ (this is related to the fact that decomposition numbers corresponding to major subsections do not vanish). 
Therefore, by \autoref{astconst}, every row of $Q_u$ appears $|Z|$ times. 
Taking only every $|Z|$-th row of $Q_u$, we obtain an indecomposable matrix $\widetilde{Q}_u\in\mathbb{Z}^{k\times l(b)}$ of rank $l(b)$ without vanishing rows such that $\widetilde{C}_u=\widetilde{Q}_u^\text{T}\widetilde{Q}_u$ and $k:=\frac{1}{|Z|}k(b)$ (see \cite[Definition~1 and Proposition~2]{SambaleC3}). Lemma~4 in \cite{SambaleC3} gives
\[\min\{|D|xC_u^{-1}x^\text{T}:0\ne x\in\mathbb{Z}^{l(b)}\}=\min\{\det(\widetilde{C}_u)x\widetilde{C}_u^{-1}x^\text{T}:0\ne x\in\mathbb{Z}^{l(b)}\}\ge l(b).\]
Hence, a result by Brauer (see \cite[Theorem~4.4]{habil}) implies the first claim. The second claim is trivial.
\end{proof}

Since every abelian coprime linear group has a regular orbit, we obtain the following (cf. \cite[Lemma~14.6]{habil}).

\begin{Corollary}
Let $B$ be a block of a finite group with abelian defect group $D$. Suppose that $I(B)$ contains an abelian subgroup of prime index or of index at most $5$. Then $k(B)\le|D|$.
\end{Corollary}

A recent paper by Keller-Yang~\cite{KellerYang} provides a dual version. 

\begin{Corollary}
Let $B$ be a block of a finite group with abelian defect group $D$. Suppose that the commutator subgroup $I(B)'$ has prime order or order at most $5$. Then $k(B)\le|D|$.
\end{Corollary}

Now we prove a result about the number of irreducible Brauer characters.

\begin{Proposition}\label{IBr}
Let $B$ be a block of a finite group with abelian defect group $D$ such that $e(B)$ is a prime. Then $l(B)\le e(B)$.
\end{Proposition}
\begin{proof}
By \cite{Watanabe1} we may assume that $\C_D(I(B))=1$. Then for every non-trivial $B$-subsection $(u,b)$ we have $l(b)=1$. Since $I(B)$ acts freely on $D$, the number of conjugacy classes of these subsections is $(|D|-1)/e(B)$. In particular, 
\[k(B)=\frac{|D|-1}{e(B)}+l(B).\]
Let $C$ be the Cartan matrix of $B$. By \cite{DetCartan}, $\det(C)=|D|$. Hence, \cite[Theorem~5]{SambaleC3} implies
\[k(B)\le\frac{|D|-1}{l(B)}+l(B).\]
The claim follows.
\end{proof}

Observe that Alperin's Weight Conjecture predicts that $l(B)=e(B)$ in the situation of \autoref{IBr}. This has been shown for principal blocks in \cite{SawabeW}. Our next result covers a special case of Usami~\cite{Usami23I}. This is of interest, since the proof in case $e(B)=3$ was announced in \cite[Introduction]{UsamiZ4}, but never appeared in print (to the author's knowledge).

\begin{Theorem}\label{e7}
Let $B$ be a $2$-block of a finite group with abelian defect group $D$ such that $e(B)\le 7$. Then $B$ is perfectly isometric \textup{(}even isotypic\textup{)} to the principal $2$-block of $D\rtimes I(B)$.
\end{Theorem}
\begin{proof}
In order to determine $l(B)$, we may assume that $\C_D(I(B))=1$. In case $e(B)=1$ the block is nilpotent and the claim is well-known. Thus, let $e(B)>1$. Since $e(B)$ is odd, we must have $e(B)\in\{3,5,7\}$. In particular, \autoref{IBr} implies $l(B)\le e(B)$. 
Let $(u,b)$ be a $B$-subsection such that $u$ has order $2$. Since $l(b)=1$, the generalized decomposition numbers $d^u_{\chi\phi}$ ($\chi\in\Irr(B)$, $\IBr(b)=\{\phi\}$) form a column of $k(B)$ non-zero integers whose sum of squares equals $|D|$. By the Kessar-Malle Theorem~\cite{KessarMalle} about Brauer's Height Zero Conjecture, we know that all irreducible characters in $B$ have height $0$. It follows that the numbers $d^u_{\chi\phi}$ are odd (see for example \cite[Lemma~1.38]{habil}). Hence, $k(B)\equiv |D|\pmod{8}$. 
Since $1\le l(B)\le e(B)\le 7$ and
\[\frac{|D|-1}{e(B)}+l(B)=k(B)\equiv |D|\equiv \frac{|D|-1}{e(B)}+e(B)\pmod{8},\]
we get $l(B)=e(B)$. Now the claim follows from the main theorem of \cite{WatanabeCycFoc}.
%
%
\end{proof}

We remark that the Cartan matrix of $B$ in the situation of \autoref{e7} is given by \[\lvert\C_D(I(B))\rvert\biggl(\frac{|[D,I(B)]|-1}{e(B)}+\delta_{ij}\biggr)_{i,j=1}^{e(B)}\] 
up to basic sets where $\delta_{ij}$ is the Kronecker delta (see \cite[Proposition~6]{SambaleC3}).

Now we present an extended version of \cite[Theorem~13.2]{habil} in the spirit of \cite{Watanabe1}.

\begin{Proposition}\label{E16}
Let $B$ be a $2$-block of a finite group with abelian defect group $D$ such that $\lvert[D,I(B)]\rvert\le 16$. Then one of the following holds:
\begin{enumerate}[(i)]
\item $B$ is nilpotent. Then $e(B)=l(B)=1$ and $k(B)=|D|$.
\item $e(B)=l(B)=3$, $\lvert[D,I(B)]\rvert=4$, $k(B)=|D|$ and the Cartan matrix of $B$ is $\frac{1}{4}|D|(1+\delta_{ij})$
up to basic sets.
\item $e(B)=l(B)=3$, $\lvert[D,I(B)]\rvert=16$, $k(B)=\frac{1}{2}|D|$ and the Cartan matrix of $B$ is $\frac{1}{16}|D|(5+\delta_{ij})$
up to basic sets.
\item $e(B)=l(B)=5$, $k(B)=\frac{1}{2}|D|$ and the Cartan matrix of $B$ is $\frac{1}{16}|D|(3+\delta_{ij})$
up to basic sets.
\item $e(B)=l(B)=7$, $k(B)=|D|$ and the Cartan matrix of $B$ is
$\frac{1}{8}|D|(1+\delta_{ij})$ up to basic sets.
\item\label{e9a} $e(B)=l(B)=9$, $k(B)=|D|$ and the Cartan matrix of $B$ is $\frac{1}{16}|D|(1+\delta_{ij})_{i,j=1}^3\otimes(1+\delta_{ij})_{i,j=1}^3$
up to basic sets where $\otimes$ denotes the Kronecker product.
\item $e(B)=9$, $l(B)=1$ and $k(B)=\frac{1}{2}|D|$.
\item\label{e15a} $e(B)=l(B)=15$, $k(B)=|D|$ and the Cartan matrix of $B$ is $\frac{1}{16}|D|(1+\delta_{ij})$ up to basic sets.
\item $e(B)=21$, $l(B)=5$, $k(B)=|D|$ and the Cartan matrix of $B$ is
\[\frac{|D|}{8}\begin{pmatrix}2&.&.&.&1\\.&2&.&.&1\\.&.&2&.&1\\.&.&.&2&1\\1&1&1&1&4\end{pmatrix}\]
up to basic sets.
\end{enumerate}
\end{Proposition}
\begin{proof}
In case $[D,I(B)]=1$, the block $B$ is nilpotent and the first case applies. Thus, we may assume that $B$ is non-nilpotent for the rest of the proof.
Since the action of $I(B)$ on $[D,I(B)]$ is coprime, we need to discuss the following cases $[D,I(B)]\in\{Z_2^2,Z_2^3,Z_2^4,Z_4^2\}$. 
The different actions on these groups can be determined easily. As usual $D=\C_D(I(B))\times[D,I(B)]$. Let $Z:=\C_D(I(B))$, and let $b_Z$ be a Brauer correspondent of $B$ in $\C_G(Z)$. Then by \cite{Watanabe1} (applied repeatedly), $l(B)=l(b_Z)$ and $k(B)=k(b_Z)$. Moreover, $b_Z$ dominates a block $\overline{b_Z}$ of $\C_G(Z)/Z$ with defect group $D/Z\cong[D,I(B)]$ and $l(\overline{b_Z})=l(b_Z)$. Using \cite[Theorems~8.1, 13.1 and 13.2]{habil} and \autoref{kBdom} it is easy to determine $l(B)=l(\overline{b_Z})$ and $k(B)=|Z|k(\overline{b_Z})$. Therefore, it remains to compute the Cartan matrix of $B$. 

The case $e(B)\le 7$ is covered by \autoref{e7} and the subsequent remark. The same argument also works for $e(B)=15$, since here $I(B)$ acts freely on $[D,I(B)]$.
Therefore, we may assume that $[D,I(B)]\in\{Z_2^3,Z_2^4\}$.
We explain our general method for these cases. Let $\mathcal{R}$ be a set of representatives for the $I(B)$-conjugacy classes of $[D,I(B)]$. For $x\in\mathcal{R}$ let $Q_x$ be the part of the generalized decomposition matrix of $B$ corresponding to the subsection $(x,b_x)$. Then by \autoref{astconst} (together with \cite[Theorem~2]{RobinsonFocal}), every row of $Q_x$ appears $|Z|$ times. This holds in particular for the ordinary decomposition matrix $Q_1$. Hence, in order to compute $Q_1$ we may divide the Cartan matrices $C_x$ of $b_x$ by $|Z|$. So, let $\widetilde{C}_x:=\frac{1}{|Z|}C_x$ for $1\ne x\in\mathcal{R}$. Since $x$ has order at most $2$, the matrices $Q_x$ are all integral. Assume that we have found matrices $\widetilde{Q}_x\in\mathbb{Z}^{k(\overline{b_Z})\times l(b_x)}$ ($1\ne x\in\mathcal{R}$) such that
\[\widetilde{Q}_x^\text{T}\widetilde{Q}_y=\begin{cases}\widetilde{C}_x&\text{if }x=y,\\0&\text{if }x\ne y\end{cases}\]
for $x,y\in\mathcal{R}\setminus\{1\}$. This means we are actually constructing the generalized decomposition matrix of $\overline{b_Z}$. Let
\[\Gamma:=\{v\in\mathbb{Z}^{k(\overline{b_Z})}:v\widetilde{Q}_x=0\in\mathbb{Z}^{l(b_x)}\ \forall x\in\mathcal{R}\setminus\{1\}\}.\]
We choose a basis for the $\mathbb{Z}$-module $\Gamma$ and we write the basis vectors as columns of a matrix $\widetilde{Q}_1\in\mathbb{Z}^{k(\overline{b_Z})\times l(B)}$ (cf. \cite[Section~4.2]{habil}). Finally, set 
\[Q_1:=\begin{pmatrix}
\widetilde{Q}_1\\\vdots\\\widetilde{Q}_1
\end{pmatrix}\in\mathbb{Z}^{k(B)\times l(B)}.\]
Then the orthogonality relations for the group $[D,I(B)]$ guarantee that $Q_1$ is orthogonal to any column of generalized decomposition numbers corresponding to a non-trivial subsection (provided a suitable ordering of $\Irr(B)$). Since the elementary divisors of $Q_1$ are equal $1$, the Cartan matrix of $B$ is given by $Q_1^\text{T}Q_1$ up to basic sets. 

Now we have to deal with the various cases according to the action of $I(B)$ on $[D,I(B)]$. As mentioned above, we may assume that $e(B)\in\{9,21\}$. If $[D,I(B)]\cong Z_2^3$, it follows that $e(B)=21$ and $I(B)\cong Z_7\rtimes Z_3$. Here there is only one matrix $\widetilde{C}_x$ for $x\in\mathcal{R}\setminus\{1\}$ given by $\widetilde{C}_x=2(1+\delta_{ij})_{i,j=1}^3$ (see \autoref{e7} and the subsequent remark).
Since $k(\overline{b_Z})=8$ (see \cite[Theorem~13.1]{habil}), there is essentially only one choice for $\widetilde{Q}_x$, namely
\[\widetilde{Q}_x=\begin{pmatrix}
1&1&1&1&.&.&.&.\\
1&1&.&.&1&1&.&.\\
1&1&.&.&.&.&1&1
\end{pmatrix}^\text{T}.\]
This makes it easy to compute $\widetilde{Q}_1$ and $Q_1^\text{T}Q_1$. Now let $[D,I(B)]\cong Z_2^4$ and $e(B)=9$. Then $I(B)$ is elementary abelian. This needs some extra arguments. In the proof of \cite[Theorem~13.7]{habil} we have
already enumerated the matrices $\widetilde{Q}_x$ for $1\ne x\in\mathcal{R}$. Doing so, we have constructed a list of ten possible Cartan matrices $\widetilde{C}_1$ and it remained open if they coincide up to basic sets. One can show with 
Magma~\cite{Magma} that the corresponding quadratic forms all have minimum $4$. One can show further that the ten quadratic forms all lie in the same genus. Moreover, this genus consists of just ten isometry classes of quadratic forms. However, a quadratic form from only one of these isometry classes has minimum $4$ (all the other possibilities yield a minimum of $2$). This shows that in fact all ten Cartan matrices differ only by basic sets. We have given the Cartan matrix of the principal block of the group $D\rtimes I(B)\cong\C_D(I(B))\times A_4^2$ where $A_4$ is the alternating group of degree $4$. 
\end{proof}

We note that part \eqref{e15a} of \autoref{E16} relies on the classification of the finite simple groups.

The argument of \autoref{E16} also works for other situations. However, it is not clear if in general $B$ and $b_z$ for $z\in\C_D(I(B))$ have the same Cartan matrix up to basic sets. This depends on the question whether the knowledge of the number $l(B)$ and the Cartan matrices of $b_x$ for $1\ne x\in [D,I(B)]$ determine the Cartan matrix of $B$.
It is conjectured in general that the blocks $B$ and $b_z$ are perfectly isometric or even Morita equivalent (see for example \cite{KLtame,KuelshammerOkuyama}). 

\begin{Corollary}\label{cor2}
Let $B$ be a $2$-block of a finite group with abelian defect group $D$. Suppose that there is an element $u\in D$ such that $|[D,\C_{I(B)}(u)]|\le 16$. Then $k(B)\le|D|$.
\end{Corollary}
\begin{proof}
Let $(u,b)$ be a $B$-subsection. Then \autoref{E16} applies for $b$. If $9\ne e(b)\ne 21$, then the action of $I(b)$ on $[D,I(b)]$ is free, and the claim follows from \autoref{freeact}. Now let $e(b)=21$. 
Here one can apply \cite[Theorem~4.2]{habil} with the quadratic form corresponding to the positive definite matrix
\[\frac{1}{2}\begin{pmatrix}2&1&.&.&-1\\1&2&.&.&-1\\.&.&2&.&-1\\.&.&.&2&-1\\-1&-1&-1&-1&2\end{pmatrix}.\]
Finally, let $e(b)=9$. Let $C_b$ be the Cartan matrix of $b$ given by \autoref{E16}. In order to apply \cite[Theorem~4.4]{habil} we consider the quadratic form corresponding to the matrix $|D|C_b^{-1}$. For this let $M:=(1+\delta_{ij})_{i,j=1}^3$. Then $4M^{-1}=(-1+4\delta_{ij})$. For $0\ne x=(x_1,x_2,x_3)\in\mathbb{Z}^3$ we have
\[4xM^{-1}x^\text{T}=x_1^2+x_2^2+x_3^2+(x_1-x_2)^2+(x_1-x_3)^2+(x_2-x_3)^2\ge 3.\]
This shows that $\min\{4xM^{-1}x^\text{T}:0\ne x\in\mathbb{Z}^3\}=3$. In general the minimum of a tensor product of quadratic forms does not need to coincide with the product of the minima of its factors. However, in this case it is true by \cite[Theorem~7.1.1]{Kitaoka}. For the convenience of the reader, we give an elementary argument. First observe that $|D|C_b^{-1}=16(M\otimes M)^{-1}=4M^{-1}\otimes4M^{-1}$. Now let $0\ne x=(x_1,x_2,x_3)\in\mathbb{Z}^9$ with $x_i\in\mathbb{Z}^3$. Then
\[16x(M\otimes M)^{-1}x^\text{T}=\sum_{i=1}^3{4x_iM^{-1}x_i^\text{T}}+\sum_{i<j}{4(x_i-x_j)M^{-1}(x_i-x_j)^\text{T}}\ge 3\min\{4yM^{-1}y^\text{T}:0\ne y\in\mathbb{Z}^3\}\ge 9.\]
Hence, $\min\{x|D|C_b^{-1}x^\text{T}:0\ne x\in\mathbb{Z}^9\}=9$, and the claim follows from \cite[Theorem~4.4]{habil}.
\end{proof}

\begin{Proposition}\label{p3}
Let $B$ be a $3$-block of a finite group with abelian defect group $D$. Suppose that there is an element $u\in D$ such that $|[D,\C_{I(B)}(u)]|\le 9$. Then $k(B)\le|D|$.
\end{Proposition}
\begin{proof}
By \autoref{freeact} we may assume that $[D,\C_{I(B)}(u)]$ is elementary abelian of order $9$. Let $(u,b)$ be a $B$-subsection. Then $I(b)\cong\C_{I(B)}(u)\le\Aut([D,\C_{I(B)}(u)])\cong\GL(2,3)$. Therefore, $I(b)$ lies in a Sylow $2$-subgroup of $\GL(2,3)$ which is isomorphic to the semidihedral group $SD_{16}$ of order $16$.
By \cite[Lemma~14.5]{habil}, we may assume that $e(b)\ge 8$. If $I(b)\in\{Z_8,Q_8\}$ where $Q_8$ is the quaternion group of order $8$, then the action of $I(b)$ on $[D,I(b)]$ is free (even regular). Hence, these cases are handled by \autoref{freeact}. It remains to deal with the cases $I(b)\in\{D_8,SD_{16}\}$. 
In order to do so, we may consider a block $\overline{b}$ with defect group $Z_3^2$ and inertial quotient $I(b)$.
The numbers $k(\overline{b})$ and $l(\overline{b})$ were determined in \cite{Kiyota,WatanabeSD16}. The case $l(b)=l(\overline{b})=2$ can be ignored by \cite[Theorem~4.9]{habil}. Hence, we have $k(\overline{b})=9$ and $l(\overline{b})\in\{5,7\}$ according to the two possibilities for $I(b)$. 
In the proof of \cite[Theorem~13.7]{habil} we have computed the possible Cartan matrices for $\overline{b}$:
\begin{align*}
\begin{pmatrix}
3&.&1&.&1\\
.&3&1&.&1\\
1&1&3&1&.\\
.&.&1&3&1\\
1&1&.&1&3
\end{pmatrix}&&\text{or}&&\begin{pmatrix}
2&1&.&.&.&.&1\\
1&2&.&.&.&.&1\\
.&.&2&1&.&.&1\\
.&.&1&2&.&.&1\\
.&.&.&.&2&1&1\\
.&.&.&.&1&2&1\\
1&1&1&1&1&1&3
\end{pmatrix}.
\end{align*}
Since the construction of these matrices was carried out by enumerating the generalized decomposition numbers as in the proof of \autoref{E16}, the Cartan matrix of $b$ is just a scalar multiple of one of these matrices. Now we can apply \cite[Theorem~4.2]{habil} with the quadratic form corresponding to the positive definite matrix 
\begin{align*}
\frac{1}{2}\begin{pmatrix}
2&.&-1&.&-1\\
.&2&-1&1&-1\\
-1&-1&2&-1&1\\
.&1&-1&2&-1\\
-1&-1&1&-1&2
\end{pmatrix}&&\text{or}&&
\frac{1}{2}\begin{pmatrix}
2&-1&.&.&.&.&-1\\
-1&2&.&.&.&.&.\\
.&.&2&-1&.&.&-1\\
.&.&-1&2&.&1&.\\
.&.&.&.&2&-1&-1\\
.&.&.&1&-1&2&.\\
-1&.&-1&.&-1&.&2
\end{pmatrix}
\end{align*}
respectively. This completes the proof.
\end{proof}

The following result about regular orbits under coprime actions might be of general interest.

\begin{Proposition}\label{regorb}
Let $P$ be an abelian $p$-group such that $\Omega(P)\subseteq\Phi(P)$. Then every $p'$-automorphism group of $P$ has a regular orbit on $P$.
\end{Proposition}
\begin{proof}
Let $A\le\Aut(P)$ be a $p'$-group. 
We may decompose $P=\bigoplus_{i=1}^n{P_i}$ into indecomposable $A$-invariant summands $P_i$. Since $\bigoplus_{i=1}^n{\Omega(P_i)}=\Omega(P)\subseteq\Phi(P)=\bigoplus_{i=1}^n{\Phi(P_i)}$, we have $\Omega(P_i)\subseteq\Phi(P_i)$ for $i=1,\ldots,n$. Suppose that we have found $x_i\in P_i$ such that $\C_A(x)=\C_A(P_i)$. Then \[\C_A(x)=\C_A(x_1)\cap\ldots\cap\C_A(x_n)=\C_A(P_1)\cap\ldots\cap\C_A(P_n)=\C_A(P)=1\] 
for $x:=x_1\ldots x_n$. Hence, we may assume that $A$ acts indecomposably on $P$. By \cite[Theorem~5.2.2]{Gorenstein}, $P$ is homocyclic, i.\,e. a direct product of isomorphic cyclic groups. By hypothesis, the exponent of $P$ is at least $p^2$. Without loss of generality, we may assume that this exponent is exactly $p^2$ (otherwise replace $P$ by $\Omega_2(P)$). 
Following \cite[Lemma~1.7]{Turull}, we will show that the action of $A$ on $P$ is isomorphic to the componentwise action of $A$ on $\Omega(P)\times\Omega(P)$. Let $x\Omega(P)\in P/\Omega(P)$. Since $A$ acts on $P/\Omega(P)$, we can define a subgroup $A_1:=\C_A(x\Omega(P))\le A$ which fixes $x\Omega(P)$ as a set. By \cite[8.2.1]{Kurzweil}, there exists a representative $r(x\Omega(P))$ of $x\Omega(P)$ such that $r(x\Omega(P))\in\C_P(A_1)$. Now for any $a\in A$ we set $r(^ax\Omega(P)):={^ar(x\Omega(P))}$. This is well defined, since $^ax\equiv {^bx}\pmod{\Omega(P)}$ implies $b^{-1}a\in A_1$ and $^ar(x\Omega(P))={^br(x\Omega(P))}$ for $a,b\in A$. Repeating this with the other orbits of cosets we end up with an $A$-invariant transversal $\mathcal{R}$ for $P/\Omega(P)$. Now let
\begin{align*}
\phi:P&\longrightarrow\Omega(P)\times\Omega(P),\\
\widetilde{x}y&\longmapsto (\widetilde{x}^p,y)\hspace{1cm}(\widetilde{x}\in\mathcal{R},\ y\in\Omega(P)).
\end{align*}
It is easy to see that $\phi$ is a bijection and
\[^a\phi(\widetilde{x}y)=({^a\widetilde{x}^p},{^ay})=\phi({^a\widetilde{x}}{^ay})=\phi({^a(\widetilde{x}y)})\]
for $a\in A$, $\widetilde{x}\in\mathcal{R}$ and $y\in\Omega(P)$. Hence, $P$ is $A$-isomorphic to $\Omega(P)\times\Omega(P)$.

By \cite{base2}, there exist $x,y\in \Omega(P)$ such that $\C_A(x)\cap\C_A(y)=1$. Hence, the $A$-orbit of $(x,y)\in \Omega(P)\times\Omega(P)$ is regular. The claim follows.
\end{proof}

We are now in a position to generalize other theorems from \cite[Chapter~14]{habil}.

\begin{Theorem}\label{elab}
Let $B$ be a block of a finite group with abelian defect group $D$ such that $D$ has no elementary abelian direct summand of order $p^3$. Then $k(B)\le|D|$.
\end{Theorem}
\begin{proof}
As in the proof of \autoref{regorb}, we can decompose $D=\bigoplus_{i=1}^n{D_i}$ into indecomposable $I(B)$-invariant summands $D_i$. Each $D_i$ is homocyclic. If $D_i$ is not elementary abelian, then we choose $x_i\in D_i$ such that $\C_{I(B)}(x_i)=\C_{I(B)}(D_i)$ by \autoref{regorb}. Now assume that $D_i$ is elementary abelian. Then by hypothesis, $|D_i|\le p^2$. Here we choose any $1\ne x_i\in D_i$. If all elementary abelian components $D_i$ have order $p$, then it is easy to see that $\C_{I(B)}(x)=1$ for $x:=x_1\ldots x_n$. In this case the claim has already been known to Brauer (see \cite[Proposition~4.7]{habil} for example). Now suppose that only $D_1$ is elementary abelian and of order $p^2$. Then $[D,\C_{I(B)}(x)]$ is cyclic, and the claim follows from \autoref{freeact}.
\end{proof}

As usual, we can say slightly more if $p$ is small.

\begin{Proposition}\label{p2}
Let $B$ be a $2$-block of a finite group with abelian defect group $D$ such that $D$ has no elementary abelian direct summand of order $2^8$. Then $k(B)\le|D|$.
\end{Proposition}
\begin{proof}
Using the arguments in the proof of \autoref{elab}, we may assume that $D$ is elementary abelian of order at most $2^7$. 
We will choose an element $x\in D$ such that $|[D,\C_{I(B)}(x)]|$ is small. 
By \autoref{cor2}, we may assume that $32\le|[D,\C_{I(B)}(x)]|<|D|$. Let $|D|=64$. If $D$ decomposes as $D=D_1\oplus D_2$ with $I(B)$-invariant subgroups $D_i$, then we can take $1\ne x_i\in D_i$ and $x:=x_1x_2$. It follows that $|[D,\C_{I(B)}(x)]|\le 16$. Hence, we may assume that $I(B)$ acts irreducibly on $D$. By the Feit-Thompson Theorem, $I(B)$ is solvable. Thus, we can use the GAP package \texttt{IRREDSOL}~\cite{GAP47} to find all possibilities for $I(B)$. It turns out that in all cases we find elements $x\in D$ such that $|[D,\C_{I(B)}(x)]|\le 16$. Finally, let $|D|=2^7$. 
Here, it can happen that $D=D_1\oplus D_2$ with irreducible $I(B)$-invariant subgroups of order $2^4$ and $2^3$ respectively. However, there is always an element $x_1\in D_1$ such that $\C_{I(B)}(x_1)=\C_{I(B)}(D_1)$. Therefore, it remains to handle the case where $I(B)$ acts irreducibly on $D$. It turns out that we only need to deal with the case $I(B)\cong Z_{127}\rtimes Z_7$ (cf. \cite[Remark 4 on p.~168]{Suprunenko}). For this case, \autoref{freeact} applies.
\end{proof}

Apart from the elementary abelian defect group of order $64$, the proof of \autoref{p2} also works for some non-abelian defect groups of order $64$. Thus, referring to the list in \cite[p. 200]{habil}, Brauer's $k(B)$-Conjecture is still open for the defect groups $\texttt{SmallGroup(64,q)}$ where
\[q\in\{134, 135, 136, 137, 138, 139, 202, 224, 229, 230, 231, 238, 239, 242, 254, 255, 257, 258, 259, 262\}.\]
Speaking of abelian defect groups for $p=2$, the next challenge is $D\cong Z_2^8$ with $I(B)\cong (Z_{31}\rtimes Z_5)\times(\Z_7\rtimes Z_3)$ acting reducibly. 

\begin{Proposition}\label{p35}
Let $p\in\{3,5\}$, and let $B$ be a $p$-block of a finite group with abelian defect group $D$ such that $D$ has no elementary abelian direct summand of order $p^4$. Then $k(B)\le|D|$.
\end{Proposition}
\begin{proof}
The case $p=3$ follows easily from \autoref{p3}. Now let $p=5$. As before, let $D=D_1\oplus D_2$ be an $I(B)$-invariant decomposition such that $D_1$ is elementary abelian and $\C_{I(B)}(x_2)=\C_{I(B)}(D_2)$ for some $x_2\in D_2$. By \autoref{elab}, we may assume that $|D_1|=p^3$. Since $I(B)/\C_{I(B)}(D_1)\le\Aut(D_1)\cong\GL(3,5)$, one can show that there is an element $x_1\in D_1$ such that $\lvert\C_{I(B)}(x_1)/\C_{I(B)}(D_1)\rvert\le 4$ or $\C_{I(B)}(x_1)/\C_{I(B)}(D_1)\cong S_3$ where $S_3$ is the symmetric group of degree $3$ (cf. \cite[proofs of 14.16 and 14.17]{habil}). Let $x:=x_1x_2$. Then \[\C_{I(B)}(x)=\C_{I(B)}(x_1)\cap\C_{I(B)}(x_2)=\C_{I(B)}(x_1)\cap\C_{I(B)}(D_2)\hookrightarrow\C_{I(B)}(x_1)/\C_{I(B)}(D_1)\]
where the inclusion comes from the canonical map $g\mapsto g\C_{I(B)}(D_1)$.
The claim follows from \cite[Lemma 14.5]{habil}.
\end{proof}

The new method does not suffice to overcome the next problems for $p\in\{3,5,7\}$ already described in \cite[Chapter~14]{habil}.

\section*{Acknowledgment}
The author thanks Thomas Breuer for discussion and providing a new version of the GAP function\linebreak \texttt{OrthogonalEmbeddings}. The author also appreciates a very helpful correspondence with Atumi Watanabe. Moreover, the author thanks the anonymous referee for pointing out an error in the proof of \autoref{fusion} in an earlier version.
This work is supported by the Carl Zeiss Foundation and the Daimler and Benz Foundation. 

\newcommand{\noopsort}[1]{}

\begin{center}
Benjamin Sambale\\
Institut für Mathematik\\
Friedrich-Schiller-Universität\\
07743 Jena\\
Germany\\
\href{mailto:benjamin.sambale@uni-jena.de}{benjamin.sambale@uni-jena.de}
\end{center}
\end{document}